\documentclass[1p]{elsarticle}
\usepackage[utf8]{inputenc}
\DeclareUnicodeCharacter{012D}{\u{\i}}
\DeclareUnicodeCharacter{012D}{\u\i}
\usepackage{textcomp} 
\usepackage{eurosym}  

\usepackage{amsmath,amssymb}
\usepackage{amstext,amsthm}
\usepackage{latexsym}
\usepackage{amsopn}
\usepackage{xspace,float}
\usepackage[justification=centering,textformat=empty,labelsep=none]{caption}
\usepackage{euscript}
\usepackage{stmaryrd} 
\usepackage{enumerate}
\usepackage[shortlabels]{enumitem}
\setlist[enumerate]{align=left,leftmargin=*,labelsep*=5pt}

\DeclareMathOperator{\sinc}{sinc}

\DeclareMathOperator{\dist}{dist}

\DeclareMathOperator{\Lip}{Lip}

\DeclareMathOperator{\supp}{supp}

\newcommand{\abs}[1]{|#1|}

\newcommand{\rez}[1]{\frac{1}{{#1}}}
\newcommand{\bgl}[1]{\begin{equation}\label{#1}}
\newcommand{\egl}{\end{equation}}
\newcommand{\ri}{\right}
\newcommand{\li}{\left}

\newcommand{\N}{\mathbb{N}}
\newcommand{\Z}{\mathbb{Z}}
\newcommand{\Zprime}{\mathbb{Z}}
\newcommand{\Zpprime}{\mathbb{Z}\setminus\{0\}}
\newcommand{\eps}{\varepsilon}
\newcommand*{\od}{_{\rm o}}
\newcommand*{\ev}{_{\rm e}}
\newcommand*{\fe}{\ensuremath{f_{\rm e}}}
\newcommand*{\fo}{\ensuremath{f_{\rm o}}}

\newcommand{\R}{\mathbb{R}}
\newcommand{\C}{\mathbb{C}}
\newcommand{\A}{\EuScript{A}}
\newcommand*\wh[1]{\widehat{#1}}

\newcommand*{\PN}[1]{\li\bracevert\!{#1}\!\ri\bracevert\!\!}
\newcommand{\un}{\ensuremath{\infty}}
\newcommand{\bin}[2]{\Big(\genfrac{}{}{0pt}{}{#1}{#2}\Big)}
\newcommand{\p}[1]{\ensuremath{( #1 )}}
\newcommand{\pbig}[1]{\ensuremath{\bigl( #1 \bigr)}}
\newcommand{\pBig}[1]{\ensuremath{\Bigl( #1 \Bigr)}}
\newcommand{\pbigg}[1]{\ensuremath{\biggl( #1 \biggr)}}

\newcommand{\bk}[1]{\ensuremath{[ #1 ]}}
\newcommand{\bkbig}[1]{\ensuremath{\bigl[ #1 \bigr]}}

\newcommand{\br}[1]{\ensuremath{\{ #1 \}}}
\newcommand{\brbig}[1]{\ensuremath{\bigl\{ #1 \bigr\}}}
\newcommand{\brBig}[1]{\ensuremath{\Bigl\{ #1 \Bigr\}}}
\newcommand{\brbigg}[1]{\ensuremath{\biggl\{ #1 \biggr\}}}

\newcommand{\absbig}[1]{\ensuremath{\bigl| #1 \bigr|}} \let\bigabs\absbig
\newcommand{\absBig}[1]{\ensuremath{\Bigl| #1 \Bigr|}} \let\Bigabs\absBig
 \let\biggabs\absbigg

\newcommand{\norm}[1]{\ensuremath{\lVert #1 \rVert}}
\newcommand{\bignorm}[1]{\big\|#1\big\|}
\newcommand{\Bignorm}[1]{\Big\|#1\Big\|}
\newcommand{\biggnorm}[1]{\bigg\|#1\bigg\|}
\newcommand{\Biggnorm}[1]{\Bigg\|#1\Bigg\|}
\newcommand{\Oh}{\mathcal O}

\newcommand{\ds}{\displaystyle}

\newcommand{\Lone}{\ensuremath{L^1(\R)}}
\newcommand{\Ltwo}{\ensuremath{L^2(\R)}}

\newcommand{\Lp}{\ensuremath{L^p(\R)}}

\newcommand{\Wrp}{\ensuremath{W^{r,p}(\R)}}

\newcommand{\Walphatwo}{\ensuremath{W\bkbig{|v|^\alpha;\Ltwo}}}
\newcommand{\Walphaone}{\ensuremath{W\bkbig{|v|^\alpha;\Lone}}}
\newcommand{\Walphap}{\ensuremath{W\bkbig{|v|^\alpha;\Lp}}}

\newcommand{\CR}{\ensuremath{\CRb}}
\newcommand{\CRn}{\ensuremath{C_0(\R)}}
\newcommand{\CRb}{\ensuremath{C_b(\R)}}

\newcommand{\dt}{\,dt}
\newcommand{\du}{\,du}
\newcommand{\dv}{\,dv}

\newcommand{\Rae}{R^{\{\alpha\}}_{2j,\eps}\kern-1pt}

\newcommand*{\Ds}[1]{D_{\rm s}^{\{#1\}}\kern-2pt}

\newcounter{zaehler}

\def\neuerlabel{(\roman{zaehler})\hfil}
\newdimen\azlabelsep
\newdimen\aztopsep
\newdimen\azitemsep
\newenvironment{aufzaehlung}[1]%
{\begin{list}{}%
{\usecounter{zaehler}%
\settowidth{\labelwidth}{#1}
\leftmargin\labelwidth
\labelsep\azlabelsep
\addtolength{\leftmargin}{\labelsep}
\topsep\aztopsep
\itemsep\azitemsep
\renewcommand{\makelabel}[1]{\ifx##1\empty\neuerlabel\else ##1\fi}}}%
{\end{list}}

\newcommand{\ie}{i.\,e.\xspace}
\newcommand{\eg}{e.\,g.\xspace}

\theoremstyle{plain}
\newtheorem{theorem}{Theorem}[section]
\newtheorem{lemma}[theorem]{Lemma}
\newtheorem{proposition}[theorem]{Proposition}
\newtheorem{corollary}[theorem]{Corollary}

\theoremstyle{definition}

\theoremstyle{remark}

\begin{document}
\begin{frontmatter}



\title{The distance between the general Poisson summation formula and that for bandlimited functions; applications to quadrature formulae%
\footnote{Dedicated to the memory of Helmut Brass (1936--2011)}}

\author[author1]{Paul L.~Butzer}
\ead{butzer@rwth-aachen.de}

\author[author2]{Gerhard~Schmeisser}
\ead{schmeisser@mi.uni-erlangen.de}

\author[author1]{Rudolf L.~Stens\corref{cor1}}
\ead{stens@mathA.rwth-aachen.de}	
 \cortext[cor1]{Corresponding author}

\address[author1]{Lehrstuhl A für Mathematik, RWTH Aachen,
		52056 Aachen, Germany}

\address[author2]{Department of Mathematics, University of Erlangen-Nuremberg,\\
		91058 Erlangen, Germany}		

\begin{abstract}
The general Poisson summation formula of harmonic analysis and analytic number theory can be regarded as a quadrature formula with remainder. The purpose of this investigation is to give estimates for this remainder based on the classical modulus of smoothness and on an appropriate metric for describing the distance of a function from a Bernstein space. Moreover, to be more flexible when measuring the smoothness of a function, we consider Riesz derivatives of fractional order. It will be shown that the use of the above metric in connection with fractional order derivatives leads to estimates for the remainder, which are best possible with respect to the order and the constants.

\end{abstract}


\end{frontmatter}

\vspace{-3.8ex}

\section{Introduction}
The general Poisson summation formula, involving a function $f$ and its Fourier transform $\wh f$, as defined in \eqref{eq_Fourier_def}, states
\begin{equation}\label{eq_Poisson_general}
  a\sum_{k\in\Z} f(ak) = \sqrt{2\pi} \sum_{k\in\Z} \wh{f}(b k)\qquad (a>0,\ b>0,\ ab=2\pi),
\end{equation}
holding under various assumptions on $f$.

If $f$ belongs to the Bernstein space $B_\sigma^1$ (see Section~\ref{sec_notations} for definitions and notations), then $\wh f(v)$ vanishes outside $(-\sigma,\sigma)$ and \eqref{eq_Poisson_general} takes the particular form  ($a=h:=2\pi/\sigma$, $b=\sigma$)
\[
  \int_{-\infty}^\infty f(t)\,dt = h \sum_{k\in\Z} f(hk)\,,
\]
which can be interpreted as an exact quadrature formula (trapezoidal rule).

When $f\notin B_\sigma^1$, then this quadrature formula is no longer exact, but one has to add a remainder term. Indeed, the general Poisson formula \eqref{eq_Poisson_general} can be restated as
\begin{equation}\label{eq_Poisson}
  \int_{-\infty}^\infty f(t)\dt = h \sum_{k\in\Z} f(hk) - \sqrt{2\pi} \sideset{}{'}\sum_{k\in\Zprime} \wh{f}(k\sigma)\,,
\end{equation}
the dash at the summation sign indicating that the term for $k=0$ is omitted.

A function space $\A$ guaranteeing the validity of \eqref{eq_Poisson} and suitable for our purposes is given by
\[
  \A:= \brbig{f\in L^1(\R)\cap \CRb; f\in \mathrm{BV}(\R) \text{ or } \wh f \in \mathrm{BV}(\R)},
\]
where $\mathrm{BV}(\R)$ denotes the set of all functions which are of bounded variation over $\R$. See \cite[Section~5.1.5]{Butzer-Nessel_1971} and \cite{Butzer-Stens_1983b}.

Using the space $\A$, we can state \eqref{eq_Poisson} more precisely as follows:

\begin{proposition}
Let $f\in\A$.  Then for $h=2\pi/\sigma>0$
\begin{equation}\label{eq_qudrature_formula}
  \int_{-\infty}^\infty f(t)\dt = h \sum_{k\in\Z} f(hk) + R_\sigma(f),
\end{equation}
where
\begin{equation}\label{eq_restglied_summe}
  R_\sigma(f) = -\sqrt{2\pi} \sideset{}{'}\sum_{k\in\Zprime} \wh{f}(k\sigma)\,.
\end{equation}
\end{proposition}
Concerning the remainder $R_\sigma(f)$, the following estimate is a particular case of our Theorem~\ref{thm_R_vs_omega_s} below. It is implicitly contained in \cite[Theorem~1]{Butzer-Stens_1983b} (with a larger constant).

\begin{proposition}
If $f\in W^{2,1}(\R)\cap \CRb$, then for $h=2\pi/\sigma$,
\begin{equation}\label{eq_main_0}
   \abs{R_\sigma(f)} \le \frac{h^2}{12}\norm{f''}_{\Lone}\qquad (\sigma>0).
\end{equation}
\end{proposition}

The factor $h^2/12$ and the occurrence of the second order derivative $f''$ are well-known from the remainder of the trapezoidal rule for finite intervals. One has, \eg,
\[
 \int_{-X}^X f(t)\dt = \frac{h}{2}\pbig{f(-X)+f(X)} + h \sum_{k=-N+1}^{N-1} f(hk) - 2X\,\frac{h^2}{12}\, f''(\xi),
\]
with $N\in\N$, $h:=X/N$ and $\xi\in [-X, X]$.

The paper is organized as follows. Section~2 is devoted to notations and the definition of the function classes needed. In Section~3 we prove some estimates for the remainder $R_\sigma(f)$ in \eqref{eq_restglied_summe}, based on the modulus of smoothness. Section~4 is concerned with fractional order derivatives, which give rise to a finer scale when measuring the smoothness of a function.

In Section~5, we introduce the notion of the distance $\dist_\un(f, B_\sigma^1)$ of a function $f$ from the Bernstein space $B^1_\sigma$, which is used in Section~6 to give more precise estimates for the remainder $R_\sigma(f)$. Indeed, we prove in Sections~7 and~8 that the distance functional $\dist_\un(f, B_\sigma^1)$ and the remainder $R_\sigma(f)$ have the same asymptotic behaviour for $\sigma \to\un$. This means that the estimates of our Theorem~\ref{thm_main_0} are best possible with respect to the order, and in Section~9 we show that even the constants are best possible. Section~10 is concerned with some numerical examples which confirm our theoretical results.

Finally, in Section~\ref{sec_Brass} we commemorate the late Professor Helmut Brass since in this paper we follow his footsteps. In particular Section~\ref{sec_Moebius} was inspired by his work.

\section{Notations}\label{sec_notations}

By $\CRb$ we denote the class of all functions $f\colon\R \to \C$ that are continuous and bounded on $\R$. Further, $\CRn$ is the subspace of those functions of $\CRb$ satisfying $\lim_{x\to\pm\infty}f(x)=0$.

In the following we always assume $p=1$ or $p=2$. The Fourier transform $\wh{f}$ of a function $f\in \Lp$ is defined by
\begin{equation}\label{eq_Fourier_def}
  \wh{f}(v):=\rez{\sqrt{2\pi}} \,\int_\R f(u) e^{-iuv}\,du\qquad(v\in\R).
\end{equation}
For $p=1$, the integral exists as an ordinary Lebesgue integral while for $p=2$ it is defined by a limiting process; see \cite[\S\,5.2.1]{Butzer-Nessel_1971}.

For $\sigma>0$, let $B_\sigma^p$ be the \emph{Bernstein space} comprising all functions $f\in\Lp$, the Fourier transform of which vanishes outside $\p{-\sigma,\sigma}$.

\subsection{Sobolev spaces}\label{sec_Sobolev}
For $r\in\N$, denote by $\mathrm{AC}_{{\rm loc}}^{r-1}(\R)$ the class of all functions that are $(r-1)$-times locally absolutely continuous on $\R$; see \cite[pp.~6--7]{Butzer-Nessel_1971}. For $1\le p < \infty$, the following function classes have been considered in Fourier analysis:
\begin{equation*}
W^{r,p}(\R) := \Big\{f\in\Lp\;:\; f=\phi \text{ a.\,e., }
    \phi\in \mathrm{AC}_{{\rm loc}}^{r-1}(\R), \, \phi^{(k)}\in \Lp,\ 0\le k\le r\Big\};
\end{equation*}
see \cite[(3.1.48)]{Butzer-Nessel_1971}. For $f\in W^{r,p}(\R)$, we write $f^{(k)}$ instead of $\phi^{(k)}$ for $k=0,\dots, r$. By endowing $W^{r,p}(\R)$ with the norm
\begin{equation}\label{eq_Sobolev_norm}
  \|f\|_{W^{r,p}(\R)} :=  \bigg\{\sum_{k=0}^r \bignorm{f^{(k)}}_{\Lp}^p\bigg\}^{1/p},
\end{equation}
we may identify it as a \emph{Sobolev space.}

In connection with Fourier transforms, the following function classes will also be of of interest. For an arbitrary complex-valued function $\psi\colon\R\to\C$ we introduce
\[
  W\bkbig{\psi(v);\Lp}:=\Big\{f\in\Lp\;:\; \psi(v) \wh{f}(v)=\wh g(v),\ g\in\Lp\Big\}.
\]
In case $p=2$ this definition can be simplified to
\[
    W\bkbig{\psi(v);\Lp} = \brBig{f\in\Ltwo\;:\; \psi(v) \wh{f}(v)\in\Ltwo},
\]
since the Fourier transform is an isometry from $\Ltwo$ onto itself.

These classes can be used to give an alternative description of the Sobolev spaces in terms of the Fourier transform. Indeed, one has (see \cite[Theorem~5.2.21]{Butzer-Nessel_1971})
\begin{equation}\label{eq_Sobolev_equivalenz}
   \Wrp =W\bkbig{(iv)^r;\Lp}=\Big\{f\in\Lp\;:\; (iv)^r \wh{f}(v)=\wh g(v),\ g\in\Lp\Big\}.
\end{equation}
Furthermore, $\wh{f^{(r)}}(v)=(iv)^r\wh f(v)=\wh g(v)$ a.\,e., where $g$ is the function on the right-hand side of \eqref{eq_Sobolev_equivalenz}.

\subsection{Lipschitz spaces}\label{sec_Lipschitz}

The modulus of smoothness of $f\in \Lone$ of order $r\in\N$ is defined by
\[
        \omega_r(f;\delta;\Lone):= \sup_{|h|\le \delta} \|\Delta_h^r f\|_{\Lone}\qquad(\delta>0),
\]
where
\begin{equation*}
       (\Delta^r_h f)(x):= \sum_{j=0}^r (-1)^{r-j} \bin{r}{j} f(x+jh)
\end{equation*}
is the forward difference of order $r$ at $x$ with increment $h$.

For $r,s\in \N$ with $s<r$ and $0<\delta<\delta_1$ one has the estimates
\begin{align}
   \omega_{r}\p{f;\delta;L^1(\R)}\le{}& \omega_{r}\p{f;\delta_1;L^1(\R)} &&\pbig{f\in \Lone},\label{eq_omega_<_0}\\[1ex]
   \omega_{r}\p{f;\delta;L^1(\R)}\le{}& 2^r\norm{f}_{\Lone} &&\pbig{f\in \Lone},\label{eq_omega_<_1}\\[1ex]
   \omega_{r}\p{f;\delta;L^1(\R)}\le{}& \delta^r\norm{f^{(r)}}_{\Lone} &&\pbig{f\in W^{r,1}(\R)},\label{eq_omega_<_2}\\[1ex]
\omega_{r}\p{f;\delta;L^1(\R)}\le{}& \delta^s\omega_{r-s}\p{f^{(s)};\delta;L^1(\R)} &&\pbig{f\in W^{s,1}(\R)}.\label{eq_omega_<_3}
\end{align}
Furthermore, if $f\in\Lone$ is such that $\int_0^\delta t^{-s-1}\omega_{r}\p{f;t;L^1(\R)}\dt<\infty$, then $ f\in W^{s,1}(\R)$ and (see \eg \cite[Chap.~2, {\S}\,7 and p.~178]{DeVore-Lorentz_1993})
\begin{equation}
\omega_{r-s}\p{f^{(s)};\delta;L^1(\R)}\le{} C_r\int_0^\delta \frac{\omega_{r}\p{f;t;L^1(\R)}}{t^{s+1}}\dt\qquad\pbig{f\in \Lone}\label{eq_omega_<_4}
\end{equation}
for a constant $C_r>0$ depending only on $r$.

The Lipschitz classes of order $\alpha$, $0<\alpha\le r$, based on the modulus $\omega_r$ are defined by
\[
  \Lip_r(\alpha)=\Lip_r(\alpha;\Lone) :=\big\{f\in \Lone\;:\;\omega_{r}(f;\delta;\Lone)=\Oh(\delta^\alpha),\ \delta\to0+\big\}.
\]

Inequality \eqref{eq_omega_<_4} in particular implies the inclusion
\[
   \Lip_r(\alpha;L^1(\R))\subset W^{1,1}(\R) \qquad (r>1;\alpha>1),
\]
which in turn implies
\begin{equation}\label{eq_Lip_in_A}
  \Lip_r(\alpha;\Lone)\cap \CRb\subset W^{1,1}(\R)\cap \CRb\subset \A \qquad(r>1;\alpha>1),
\end{equation}
where the rightmost inclusion can be found \eg in \cite[Thm.~B]{Butzer-Stens_1983b}.

\section{Estimates of the remainder based on the modulus of smoothness}

\begin{theorem}
Let $f\in\A$.  Then for each $r\in\N$ and $\sigma>0$,
\begin{equation}\label{eq_omega_r}
   \abs{R_\sigma(f)}
   \le \frac 1{2^{r-1}}\sum_{k=1}^\un\omega_r\pBig{f;\frac{\pi}{\sigma k};L^1(\R)}.
%
\end{equation}
\end{theorem}

\begin{proof}
We start with the inequality (cf.~\cite[p.~348]{Butzer-Schmeisser-Stens_2013})
\begin{equation}\label{eq_Fourier_vs_omega}
  \bigabs{\widehat f(v)} \le \frac 1{2^r\sqrt{2\pi}}\,\omega_r\Big(f;\frac{\pi}{\abs{v}};L^1(\R)\Big)  \qquad(v\ne 0),
\end{equation}
which immediately yields the desired estimate in view of

\begin{equation*}
  \abs{R_\sigma(f)} \le \sqrt{2\pi} \sum_{k\in\Z\setminus\{0\}}  \bigabs{\wh{f}(k\sigma)}
  \le \frac 1{2^{r-1}}\sum_{k=1}^\un\omega_r\pBig{f;\frac{\pi}{\sigma k};L^1(\R)}\qquad (\sigma > 0).\qedhere
\end{equation*}
\end{proof}

The infinite series in \eqref{eq_omega_r} may be divergent. If one, however, assumes that $f$ satisfies a Lipschitz condition of a certain order $\alpha > 0$, then it will be finite and tend to zero for $\sigma\to\infty$.

\begin{corollary}
a) If $f\in \Lip_r(\alpha;L^1(\R))\cap \CRb$ for some $\alpha >1$, then
\[
  R_\sigma(f)=\Oh\pbig{\sigma^{-\alpha}}\qquad (\sigma\to\un).
\]

\medskip\noindent
b) If $f\in W^{s,1}(\R)\cap \CRb$ with $f^{(s)}\in \Lip_r(\beta;L^1(\R))$ for some $s,r\in\N$ and $\beta >0$, then
\[
  R_\sigma(f)=\Oh\pbig{\sigma^{-s-\beta}}\qquad (\sigma\to\un).
\]
\end{corollary}

\begin{proof}
For the proof of a) one simply inserts the Lipschitz condition $\omega_r\p{f;\delta;L^1(\R)}=\Oh(\delta^\alpha)$ for $\delta\to 0+$ into the series, and for b) one additionally uses inequality \eqref{eq_omega_<_3} with $r$ replaced by $r+s$.
\end{proof}

If $f$ has at least a second order derivative, then the remainder can also be estimated in terms of a modulus of smoothness.

\begin{theorem}\label{thm_R_vs_omega_s}
If $f\in W^{s,1}(\R)\cap \CRb$ for some $s\in \N$, $s\ge 2$, then for each $r\in\N$ and $h = 2\pi/\sigma$,
\begin{equation}\label{eq_R_mit_modulus}
  \bigabs{R_\sigma(f)}\le h^{s}\frac{\zeta(s)}{\pi^s 2^{r+s-1}}\omega_r\pBig{f^{(s)};\frac h2;L^1(\R)}\qquad (h>0).
\end{equation}
Furthermore, one has
\begin{equation}\label{eq_R_mit_ableitung}
  \bigabs{R_\sigma(f)}\le h^{s}\frac{\zeta(s)}{\pi^s 2^{s-1}}\bignorm{f^{(s)}}_{\Lone}\qquad (h>0).
\end{equation}
%
\end{theorem}

\begin{proof}
%
Since $f\in W^{s,1}(\R)=W\bkbig{(iv)^s;\Lone}$, we have by \eqref{eq_Fourier_vs_omega},
\[
  \bigabs{\widehat f(v)}=\frac{1}{\abs{v}^s}\bigabs{(iv)^s\widehat f(v)}=\frac{1}{\abs{v}^s}\bigabs{\widehat {f^{(s)}}(v)}
  \le \frac 1{2^r\sqrt{2\pi}}\,\frac{1}{\abs{v}^s}\omega_r\Big(f^{(s)};\frac{\pi}{\abs{v}};L^1(\R)\Big)  \qquad(v\ne 0).
\]
It follows that
\begin{align*}
    \abs{R_\sigma(f)}  \le{}& \sqrt{2\pi} \sum_{k\in\Z\setminus\{0\}}  \bigabs{\wh{f}(k\sigma)}
    \le \frac 1{2^{r-1}}\sum_{k=1}^\un \frac{1}{\p{\sigma k}^s}\omega_{r}\pBig{f^{(s)};\frac{\pi}{\sigma k};L^1(\R)}\\[2ex]
    {}={}&\frac{\zeta(s)}{\pi^s 2^{r+s-1}}\pBig{\frac{2\pi}{\sigma}}^s\omega_{r}\pBig{f^{(s)};\frac{\pi}{\sigma};L^1(\R)}
    = h^s\frac{\zeta(s)}{\pi^s 2^{r+s-1}}\omega_{r}\pBig{f^{(s)};\frac{h}{2};L^1(\R)}.
\end{align*}
This is \eqref{eq_R_mit_modulus}, and \eqref{eq_R_mit_ableitung} follows by applying \eqref{eq_omega_<_1} to this inequality.
\end{proof}

Estimate \eqref{eq_R_mit_ableitung} can also be deduced from the results in \cite{Butzer-Stens_1983b}, but our approach yields a better constant. For $s=2$ we obtain \eqref{eq_main_0}.

\section{Fractional order derivatives}\label{sec_fractional_der}
In this section we are going to generalize estimate \eqref{eq_main_0} to derivatives of arbitrary order. We will consider derivatives of integer order, but also of fractional order $\alpha>1$. For our purpose we have to consider such derivatives for $f\in L^p(\R)$, $p=1$ or $p=2$.

A function $f\in L^p(\R)$, $p=1$ or $p=2$ is said to have a strong (norm) Riesz derivative of fractional order $\alpha$ with $0<\alpha < 2j$, $j\in \N$, if there exists function $g\in L^p(\R)$ such that  (\cite[Chapter~11]{Butzer-Nessel_1971}, \cite{Butzer-Schmeisser-Stens_2015})
\[
%
  \lim_{\eps\to 0+}\Biggnorm{\frac1{C_{\alpha,2j}} \int_\varepsilon^\infty \frac{\overline\Delta^{2j}_u f(\,\cdot\,)}{u^{1+\alpha}}\du-g(\,\cdot\,) }_{\Lp}=0.
\]
Then the strong Riesz derivative is defined by $\Ds\alpha f:=g$.

Above
\begin{equation}\label{eq_difference_central}
  \overline\Delta^{2j}_u f(x):=\sum_{k=0}^{2j}(-1)^k\bin {2j} k f\pbig{x+\p{j-k}u}\qquad(x,u\in\R)
\end{equation}
is the central difference of $f$ of order $2j$ at $x$ with increment $u$ and
\begin{equation}\label{eq_C}
  C_{\alpha,2j}:=(-1)^j2^{2j-\alpha} \int_0^\infty \frac{\sin^{2j}u}{u^{1+\alpha}}\,du.
\end{equation}
It should be noted that the definition of $\Ds\alpha f$ is independent of $j\in\N$. See \cite[Chapter~11]{Butzer-Nessel_1971}, \cite{Butzer-Schmeisser-Stens_2015} for the details.

%
%
%
Now to the characterization of the Riesz derivative in terms of the Fourier transform (cf.~\cite[Theorem~11.2.9]{Butzer-Nessel_1971}, \cite{Butzer-Schmeisser-Stens_2015}).
\begin{proposition}\label{prop_Riesz_existence_equiv}
The following assertions are equivalent for $f\in\Lp$, $p=1,\,2$, and $\alpha>0$:
\begin{aufzaehlung}{$(ii)$}
\item $f$ has a strong Riesz derivative $\Ds{\alpha} f$;
%
%
\item $f\in  \Walphap $, \ie, there exist a function $g\in \Lp$ with $\wh g(v)= |v|^\alpha \wh f(v)$.
\end{aufzaehlung}
In this event, $\wh{\Ds{\alpha}f}(v)=\abs{v}^\alpha\widehat f(v)$. If, in addition, $\abs{v}^\alpha\widehat f(v)\in L^1(\R)$, then
\begin{equation}\label{eq_Riesz_representation}
   \Ds{\alpha}f(t)=\frac{1}{\sqrt{2\pi}}\int_{-\un}^\un\abs{v}^\alpha\widehat f(v) e^{ivt}\dv \qquad(t\in\R).
\end{equation}
\end{proposition}

When $\alpha = r$ with $r\in \N$, the Riesz derivative can be expressed in terms of the ordinary derivative $f^{(r)}$ or Riemann derivative $D^{\bk{r}}\kern-2pt f$, because
%

\begin{equation}\label{eq_Riesz_vs_Riemann}
   \Ds r f(x)=
  \begin{cases}
    (-1)^m f^{(2m)}(x)= (-1)^m D_s^{\bk{2m}}\kern-2pt f(x), & r=2m,\\[2ex]
    (-1)^{m-1}{\widetilde f}^{(2m-1)}(x)=(-1)^{m-1}D_s^{\bk{2m-1}}\kern-2pt \widetilde f(x),& r=2m-1,
  \end{cases}
\end{equation}
where $\widetilde f$ denotes the conjugate function or Hilbert transform of $f$. See \cite[Section~3.3]{Butzer-Schmeisser-Stens_2015} for the details.

Theorem~\ref{thm_R_vs_omega_s} can be easily extended to fractional order derivatives. The proof follows by exactly the same arguments. One has only to note that $\Walphaone \subset W^{1,1}(\R)\subset \A$; see \cite[p.~416]{Butzer-Nessel_1971}, \cite{Butzer-Stens_1983b}.

\begin{theorem}\label{thm_R_vs_omega_alpha}
Let $\alpha>1$. If $f\in \Walphaone\cap\CR$, \ie, $f$ possesses a strong Riesz derivative $\Ds{\alpha}f\in \Lone$, then for each $r\in\N$ and $h = 2\pi/\sigma$,
\begin{equation}\label{eq_R_mit_modulus_alpha}
  \bigabs{R_\sigma(f)}\le h^{\alpha}\frac{\zeta(\alpha)}{\pi^\alpha 2^{r+\alpha-1}}\omega_r\pBig{\Ds{\alpha}f;\frac h2;L^1(\R)}\qquad (h>0).
\end{equation}
Furthermore, one has
\begin{equation}\label{eq_R_mit_ableitung_alpha}
  \bigabs{R_\sigma(f)}\le h^{\alpha}\frac{\zeta(\alpha)}{\pi^\alpha 2^{\alpha-1}}\bignorm{\Ds{\alpha}f}_{\Lone}\qquad (h>0).
\end{equation}
%
\end{theorem}

\section{Distance of functions from $B^1_\sigma$}
Let $G$ be the vector space of all functions $f\in \Ltwo$ having their Fourier transform in $\CRb$.
%
On the space $G$ we define a norm by
\[
       \PN{f}_\un := \bignorm{\wh f\,}_{\CR} = \sup_{v\in\R} \bigabs{\wh f(v)},
\]
which induces the metric
\begin{equation*}
\dist_\un(f,g):= \PN{f-g}_\un.
\end{equation*}

We now determine the distance of a function $f\in G$ from the Bernstein space $B_\sigma^1$, defined
by
\[
 \dist_\un(f, B_\sigma^1) := \inf_{g\in B_\sigma^1} \PN{f-g}_\un.
\]
\begin{proposition}\label{prop_prop4.1_neu}
a) Let $f\in G$. Then,
\[
 \dist_\un(f, B_\sigma^1)= \sup_{|v|\ge \sigma}\bigabs{\wh f(v)}.
 \]
%
%
b) Let $f\in \Ltwo$ with $v^s\wh f(v)\in \Ltwo\cap \CR$ for some $s\in \N$, then $f$ has a derivative of order $s$ in  $\Ltwo$, and
\[
 \dist_\un(f^{(s)}, B_\sigma^1)= \sup_{|v|\ge \sigma} \absbig{v^s\wh f(v)}.
 \]
\noindent
c) Let $f\in\Ltwo$ with $\abs{v}^\alpha\wh f(v)\in \Ltwo\cap \CR$ for some $\alpha>0$. Then $f$ has a Riesz derivative $\Ds \alpha f$  in  $\Ltwo$, and
\[
 \dist_\un(\Ds \alpha f, B_\sigma^1)= \sup_{|v|\ge \sigma} \abs{v}^\alpha\absbig{\wh f(v)}.
 \]
%
\end{proposition}

\begin{proof} Clearly,
\begin{equation*}
\dist_\un(f, B_\sigma^1)
   = \inf_{g\in B_\sigma^1}\brbigg{\sup_{v\in\R}\bigabs{\wh f(v)-\wh{g}(v)}} \ge \inf_{g\in B_\sigma^1}\brbigg{\sup_{\abs v\ge \sigma}\bigabs{\wh f(v)-\wh{g}(v)}}
   = \sup_{\abs v\ge \sigma}\bigabs{\wh f(v)},
\end{equation*}
the latter equality holding in view of $\wh g(v)=0$ for $\abs v >\sigma$.

Conversely, for $t>0$ let $w_t(x):= (2t)^{-1/2}\exp(-x^2/4t)$, $x\in\R$ be the kernel of Gauß-Weierstraß having the Fourier transform $\wh w_t(v):= \exp(-tv^2)$, $v\in\R$. Now define
\[
    f_t(x):=\frac1{\sqrt{2\pi}}\int_{-\infty}^\infty(\wh f*w_t)(v)e^{ivx}\,dv= f(x)\wh w_t(x) \qquad(x\in\R),
\]
where $*$ denotes the Fourier convolution.
Since $f$ and $\wh w_t$ both belong to $\Ltwo$, it follows by Hölder's inequality that $f_t\in L^1(\R)$, and with aid of the convolution theorem we deduce  $\wh f_t=\wh f*w_t$;  see \cite[pp.~4, 125, 189]{Butzer-Nessel_1971}.

Next, for $0<\eta<\sigma-\eta$ let
\[
 \chi_\eta(x):=\sqrt{\frac2\pi}(\sigma-\eta)\sinc\Big(\frac{\sigma-\eta}{\pi}x\Big)\sinc\Big(\frac{\eta}{\pi}x\Big)
 \qquad(x\in\R),
\]
the Fourier transform of which is given by
\[
  \wh\chi_\eta(v)=
  \begin{cases}
  1, & |v|\le \sigma-\eta,\\
  0, &|v|\ge \sigma,\\
  \eta^{-1}(v+\sigma),\quad & -\sigma<v<-\sigma+\eta,\\
  \eta^{-1}(\sigma-v),\quad & \sigma-\eta <v<\sigma.
  \end{cases}
\]
In particular, there holds $0\le \wh\chi_\eta(v)\le 1$ for all $v\in\R$.

Setting now $g_{t,\eta}:=f_t*\chi_\eta$, then $g_{t,\eta}\in L^1(\R)$ with $\wh g_{t,\eta}=\wh f_t\cdot\wh\chi_\eta$, showing that $g_{t,\eta}\in B^1_\sigma$.
It follows that
\begin{align*}
 \dist_\un(f,B^1_\sigma) & {}\le \max\brbigg{\sup_{\abs{v}\le\sigma-\eta} \bigabs{\wh f(v)-\wh{g}_{t,\eta}(v)},\sup_{\abs{v}\ge\sigma-\eta} \bigabs{\wh f(v)-\wh{g}_{t,\eta}(v)}} =: \max\br{s_1,s_2},
\end{align*}
say. Since $\wh g_{t,\eta}(v)=\wh f_t(v)\cdot\wh\chi_\eta(v)=\wh f_t(v)$ for $\abs v\le \sigma-\eta$,  we have
\begin{equation*}
  s_1 = \sup_{\abs{v}\le\sigma-\eta} \bigabs{\wh f(v)-\wh f_t(v)},
\end{equation*}
and, concerning $s_2$, there holds
\begin{align*}
  s_2\le{}& \sup_{\abs v\ge \sigma-\eta}\Bigabs{\wh f(v)\bkbig{1-\wh\chi_\eta(v)}}+\sup_{\abs v\ge \sigma-\eta}\Bigabs{\bkbig{\wh f(v)-\wh f_t(v)}\wh\chi_\eta(v)}\\[2ex]
 {}\le{}& \sup_{\abs v\ge \sigma-\eta}\bigabs{\wh f(v)}+\sup_{\sigma-\eta\le \abs v\le \sigma}\bigabs{\wh f(v)-\wh f_t(v)}=:s_{2,1}+s_{2,2}.
\end{align*}
%

Now, let $\eps>0$ be arbitrary. Since $\wh f $ is bounded and continuous, there exists $t>0$ such that $s_1$ and $s_{2,2}$ become smaller than $\eps/2$ (cf.~\cite[Cor.~3.1.13 and Problem~3.1.1\,(iv)]{Butzer-Nessel_1971}). Furthermore, one can choose $\eta > 0$ so small that $s_{2,1}\le \sup_{\abs v\ge \sigma}\bigabs{\wh f(v)} +\eps/2$.

It follows that
\[
  \dist_\un(f,B^1_\sigma)\le \max\br{s_1,s_{2,1}+s_{2,2}}\le \sup_{\abs v\ge \sigma}\bigabs{\wh f(v)}+ \eps.
\]
This completes the proof of part~a).

Since the proofs of parts~b) and c) follow by exactly the same arguments, we confine ourselves to the latter one.

Since the Fourier transform is a surjection on $\Ltwo$, the assumptions imply $f\in \Walphatwo$. Now we have by Proposition~\ref{prop_Riesz_existence_equiv} that the derivative $\Ds \alpha f$ exists (as an element of $\Ltwo$) and there holds $\wh{\Ds{\alpha}f}(v)=\abs{v}^\alpha\widehat f(v)\in \CRb$. Hence $\Ds \alpha f\in G$, and we can apply part~a) with $f$ replaced by $\Ds\alpha f$, and $\wh f(v)$ by $\wh{\Ds\alpha f}(v)=\abs{v}^\alpha\wh f(v)$.
\end{proof}

\section{Estimates of the remainder in terms of the distance functional}

Using the concept of the foregoing sections, we can prove the following result.
\begin{theorem}\label{thm_main_0}
a) Let $f\in \A$ with $v^s \wh{f}(v) \in L^2(\R)\cap \CRb$ for some $s\in\N$, $ s\ge 2$. Then $f^{(s)}$ exists as an element of $\Ltwo$, and one has for $h=2\pi/\sigma$,
\begin{equation}\label{eq_main_1}
  \abs{R_\sigma(f)} \le h^s\frac{2\zeta(s)}{(2\pi)^{s-1/2}}\dist_\infty(f^{(s)}, B_\sigma^1)\qquad(\sigma>0).
%
\end{equation}
b) Let $\alpha >1$, $f\in \A$ and $\abs{v}^\alpha \wh{f}(v) \in L^2(\R)\cap \CRb$. Then $\Ds{\alpha}f$ exists as an element of $\Ltwo$, and one has for $h=2\pi/\sigma$,
\begin{equation}\label{eq_main_1_alpha}
  \abs{R_\sigma(f)} \le h^\alpha\frac{2\zeta(\alpha)}{(2\pi)^{\alpha-1/2}}\dist_\infty(\Ds{\alpha}f, B_\sigma^1)\qquad(\sigma>0).
%
\end{equation}
%
%
\end{theorem}

\begin{proof} a) First we note that $f\in\mathcal A$ implies $f\in L^1(\R)\cap \CRb\subset \Ltwo$ and we also have $\abs{v}^\alpha \wh{f}(v) \in \Ltwo$. Hence the derivative $f^{(s)}$ exists as an element of $\Ltwo$ by Proposition~\ref{prop_prop4.1_neu}\,b). It follows that
%
%
\begin{align*}
\abs{R_\sigma(f)} \le {} & \sqrt{2\pi}\,\sum_{k\in\Z\setminus\{0\}} \abs{\wh{f}(k\sigma)}
      = \sqrt{2\pi}\,\sideset{}{'}\sum_{k\in\Zprime} \frac{\abs{k\sigma}^s \abs{\wh{f}(k\sigma)}}{\abs{k\sigma}^s}\\[2ex]
%
%
{}\le{} & \sqrt{2\pi} \sup_{\abs{v}\ge\sigma} \brBig{\abs{v}^s \absbig{\wh{f}(v)}}\,\frac{2}{\sigma^s} \,
\sum_{k=1}^\infty \rez{k^s} = 2\sqrt{2\pi}\, \frac{\zeta(s)}{\sigma^s}\, \sup_{\abs{v}\ge\sigma} \brBig{\abs{v}^s \absbig{\wh{f}(v)}}\,.
\end{align*}

Applying Proposition~\ref{prop_prop4.1_neu}\,b) once more, we obtain \eqref{eq_main_1}. Part~b) follows along the same lines using Proposition~\ref{prop_prop4.1_neu}\,c).
\end{proof}

\section{An equivalence theorem by means of the Möbius inversion}\label{sec_Moebius}

Split $f$ as $f= f\ev + f\od $, where
\[
  \fe(t) := \frac{f(t)+f(-t)}{2} \quad \text{ and } \quad \fo(t) := \frac{f(t) - f(-t)}{2}
\]
are the even and the odd part of $f$, respectively.

Clearly, $R_\sigma(f\od)=0$ and so $R_\sigma(f)=R_\sigma(f\ev)$. It is known that $\wh{\fe}(v)= \wh{\fe}(-v)$, and hence
\begin{equation}\label{eq_remainder_f_even}
  R_\sigma(\fe) = -2\sqrt{2\pi}\, \sum_{k=1}^\infty \wh{\fe}(k\sigma).
\end{equation}
Lyness \cite{Lyness_1970}, Brass \cite{Brass_1978} and Loxton-Sanders \cite{Loxton-Sanders_1979} were the first to use (independently) the Möbius inversion formula for deducing properties of a function $f$ from its remainders in the trapezoidal rule. We shall now also employ that technique.

The Möbius function $\mu\,:\, \N \longrightarrow \{-1, 0, 1\}$ is defined by
\[
   \mu(k)=
   \begin{cases}
      1, &  k=1,\\
 (-1)^n, &  k=p_1 \cdots p_n \hbox{ with distinct primes } p_1, \dots, p_n,\\
      0, &  k \text{ is divisible by a square of a prime}.
   \end{cases}
\]

From \cite[p.~19]{Loxton-Sanders_1980} we cite:
\begin{corollary}\label{cor_Loxton}
The following assertions are equivalent:
\begin{aufzaehlung}{$(ii)$}
\item $\ds g(n)=\sum_{m=1}^\un f(mn)$, $n\in\N$,\\[1ex]
with\\[1.5ex]
$\ds\sum_{n=1}^\un n^\eps\abs{f(n)}<\un$ for some $\eps>0$;
\item $\ds f(n)=\sum_{m=1}^\un \mu(m)g(mn)$, $n\in\N$,\\[1ex]
with\\[1.5ex]
$\ds\sum_{n=1}^\un n^\eps\abs{g(n)}<\un$ for some $\eps>0$.
\end{aufzaehlung}
\end{corollary}

The next proposition provides the Möbius inversion of \eqref{eq_remainder_f_even}, \ie the deduction of $\widehat f_e$ from \eqref{eq_remainder_f_even}.
Indeed, choosing in the above corollary $g(n):= R_{\sigma n}\pbig{\widehat{\fe}}$ and $f(n):=-2\sqrt{2\pi}\widehat{\fe}(n\sigma)$, then \eqref{eq_remainder_f_even} with $\sigma$ replaced by $\sigma n$ is just $g(n)=\sum_{m=1}^\un f(mn)$. Hence one obtains from Corollary~\ref{cor_Loxton}:

\begin{proposition}\label{prop_Loxton}
Let $f\in \A$ and suppose that
\begin{equation}\label{eq_prop_Loxton_1}
  \sum_{k=1}^\infty k^\eps \abs{\wh{\fe}(k\sigma)}<\infty
\end{equation}
for some $\eps>0$. Then
\begin{equation}\label{eq_prop_Loxton_2}
 \wh{f}\ev(\sigma) = -\rez{2\sqrt{2\pi}}\sum_{k=1}^\infty \mu(k)\,R_{k\sigma}(f\ev)\,.
\end{equation}
\end{proposition}

The following lemma shows  that \eqref{eq_prop_Loxton_2} holds when $f$ satisfies the hypotheses of Theorem~\ref{thm_main_0}\,b).

\begin{lemma}\label{lem1}
If $\abs{v}^\alpha \absbig{\wh{f}(v)}\in \CRb$ for some $\alpha>1$, then \eqref{eq_prop_Loxton_1} holds for $\eps\in (0, \alpha-1)$.
\end{lemma}

\begin{proof}
For $\eps\in(0, \alpha-1)$, we have
\[
\sum_{k=1}^\infty k^\eps \abs{\wh{f}\ev(k\sigma)} = \sum_{k=1}^\infty \frac{(k\sigma)^\alpha \abs{\wh{f}\ev(k\sigma)}}{\sigma^\alpha k^{\alpha-\eps}}
\le \sup_{\abs{v}\ge\sigma}\brBig{\abs{v}^\alpha \absbig{\wh{f}(v)}} \cdot \frac{\zeta(\alpha-\eps)}{\sigma^\alpha} < \infty.\qedhere
\]
\end{proof}

We are now ready for our main equivalence theorem.

\begin{theorem}\label{thm_main_3}
Let the hypotheses of Theorem~\ref{thm_main_0}\,b) be satisfied. For $t_0>0$ let $\lambda$ be a non-negative, nonincreasing function on $[t_0, \infty)$. Then the following statements are equivalent:
\begin{aufzaehlung}{$(ii)$}
\item $\dist_\infty(\Ds{\alpha}f\ev, B_\sigma^1)= \Oh(\lambda(\sigma))\quad (\sigma\to\infty)$;
\item $R_\sigma(f) = \Oh(\sigma^{-\alpha} \lambda(\sigma)) \quad (\sigma\to\infty)$.
\end{aufzaehlung}
\end{theorem}

\begin{proof}
Suppose that (i) holds. If $f$ satisfies the hypotheses of Theorem~\ref{thm_main_0}\,b), then so does $f\ev$. Therefore \eqref{eq_main_1_alpha} also holds with $f$ replaced by $f\ev$. Now, employing (i) and noting that $R_\sigma(f)=R_\sigma(f\ev)$, we readily obtain (ii) by \eqref{eq_main_1_alpha}.

Conversely, suppose that (ii) holds. Then there exists a constant $c>0$ and a $\sigma_0>t_0$ such that

\[
  \abs{R_\sigma(f\ev)}\,=\, \abs{R_\sigma(f)}\,\le\, c\, \frac{\lambda(\sigma)}{\sigma^\alpha}\qquad (\sigma\ge\sigma_0).
\]
Now Proposition~\ref{prop_Loxton} yields
\begin{align*}
\abs{\wh{f}\ev(\sigma)} &{}\le \rez{2\sqrt{2\pi}}\, \sum_{k=1}^\infty \abs{R_{k\sigma}(f\ev)} \le \frac{c}{2\sqrt{2\pi}}\, \sum_{k=1}^\infty
                                 \frac{\lambda(k\sigma)}{(k\sigma)^\alpha}\\[2ex]
&{}\le \frac{c}{2\sqrt{2\pi}}\cdot \frac{\lambda(\sigma)}{\sigma^\alpha} \sum_{k=1}^\infty \rez{k^\alpha} = \frac{c\,\zeta(\alpha)}{2\sqrt{2\pi}}\cdot
                                 \frac{\lambda(\sigma)}{\sigma^\alpha}
\end{align*}
%
%
for $\sigma\ge\sigma_0$. Thus
\[
 \bigabs{\sigma^\alpha \wh{f}\ev(\sigma)} \le \frac{c\,\zeta(\alpha)}{2\sqrt{2\pi}}\: \lambda(\sigma),
\]
which implies that
\[
  \sup_{\abs{v}\ge\sigma}\brBig{\abs{v}^\alpha \abs{\wh{f}\ev(v)}}\le \frac{c\,\zeta(\alpha)}{2\sqrt{2\pi}}\:\lambda(\sigma)\qquad(\sigma\ge\sigma_0).
\]
Now the considerations in Section~\ref{sec_fractional_der} concerning the strong Riesz derivative and its distance from $B_\sigma^1$ show that (i) holds.
\end{proof}

Theorem~\ref{thm_main_3} shows that the estimate in \eqref{eq_main_1_alpha} of Theorem~\ref{thm_main_0}  is optimal, at least for even functions, in the sense that both sides have the same order for $\sigma \to \un$. Indeed, if $f$ is even, then one has \eg for any $\alpha,\beta>0$,
\[
  R_\sigma (f)=\Oh\pbig{\sigma^{-\beta}}\iff \dist_\un(D_s^{\br{\alpha}}f,B^1_\sigma)=\Oh\pbig{\sigma^{-\alpha-\beta}}\qquad (\sigma\to\infty).
\]

Of course one could also prove an analogue of Theorem~\ref{thm_main_3} for ordinary derivatives $f^{(s)}$, $s\ge 2$.

\section{An equivalence theorem without using the Möbius inversion}
Knowing the order of convergence of $R_\sigma(f)$ as $\sigma\to\infty$, we cannot say anything about the regularity of $f\od$. In fact, due to a symmetry of its graph, $R_\sigma(f\od)=0$,  even if $f\od$ is a very erratic function analytically. However, the regularity of a function does not change under a translation of its argument, while possible symmetries with respect to the axes of the coordinate system can be destroyed. In Theorem~\ref{thm_main_3}, we can therefore extend statement (i) to $f$ (instead of $f\ev$), if in statement (ii) we require that the order of convergence does not change under a translation. We only need translations within the step size $h$.

Let $f_\tau:= f(\tau + \cdot)$. It is known that $\wh{f}_\tau(v)=\wh{f}(v)e^{i\tau v}$. Therefore, by \eqref{eq_restglied_summe},
\[
  R_\sigma(f_\tau)= -\sqrt{2\pi} \sideset{}{'}\sum_{k\in\Zprime} \wh{f}(k\sigma) e^{i\tau k\sigma}.
\]
Setting $x:=\tau\sigma$, we see that the right-hand side is a Fourier series in $x$. If $\tau$ runs from $-\pi/\sigma$ to $\pi/\sigma$, then $x$ runs through a period. Thus, if $\abs{R_\sigma(f_\tau)}\le c \sigma^{-\alpha} \lambda(\sigma)$ for $\tau\in [-\pi/\sigma, \pi/\sigma]$, then as a consequence of  Parseval's equation,
\[
  2\pi \sideset{}{'}\sum_{k\in\Zprime} \abs{\wh{f}(k\sigma)}^2 \le \frac{c^2 \lambda^2(\sigma)}{\sigma^{2\alpha}}.
\]
This implies -- without the need of Möbius inversion -- that
\[
  \abs{\wh{f}(\sigma)} \le \frac{c}{\sqrt{2\pi}}\cdot \frac{\lambda(\sigma)}{\sigma^{\alpha}}.
\]

These considerations show that the following variant of Theorem~\ref{thm_main_3} holds.

\begin{theorem}\label{thm3}
Let the hypotheses of Theorem~\ref{thm_main_0}\,b) be satisfied. For $t_0>0$ let $\lambda$ be a non-negative, nonincreasing function on $[t_0, \infty)$.
Then the following statements are equivalent:
\begin{aufzaehlung}{$(ii)$}
\item $\dist_\infty(\Ds{\alpha}f, B_\sigma^1) = \Oh(\lambda(\sigma))\quad (\sigma\to\infty)$;
\item $R_\sigma(f(\tau + \cdot)) =\Oh(\sigma^{-\alpha} \lambda(\sigma)) \quad \p{\sigma\to\infty$ uniformly for $-\frac{\pi}{\sigma}\le \tau\le \frac{\pi}{\sigma}}$.
\end{aufzaehlung}
\end{theorem}

Again it is possible to provide an analogue of Theorem~\ref{thm3} for ordinary derivatives $f^{(s)}$, $s\ge 2$.

Theorems~\ref{thm_main_3} and \ref{thm3} show in particular that the distance functional provides the right measure for estimating the remainder in the trapezoidal rule \eqref{eq_qudrature_formula}. Indeed, the orders for $h\to 0+$ or, equivalently, $\sigma\to\un$ on the right-hand side of estimates \eqref{eq_main_1}, \eqref{eq_main_1_alpha} coincide exactly with the orders of the remainder in \eqref{eq_qudrature_formula}. Concerning the order of remainder of the composite trapezoidal rule on compact intervals, the reader is referred to \cite{Brass_1979}.

In the following section we go even a step further. We will show that also the constants in \eqref{eq_main_1} and \eqref{eq_main_1_alpha} cannot be improved.

\section{Sharpness of Theorem~\ref{thm_main_0}}
The aim of this section is to show that the estimates of Theorem~\ref{thm_main_0} are best possible regarding the constants . To this end we construct an extremal function for which there holds equality in \eqref{eq_main_1} and \eqref{eq_main_1_alpha}, respectively, at least for one $\sigma>0$.

\begin{lemma}\label{prop_existence_extremal}
For each $\alpha>1$ there exist a function $\phi_\alpha\in \A$ such that $\abs{v}^\alpha \wh{\phi_\alpha}(v)\in \Ltwo\cap\CRb$ with $\abs{k}^\alpha\wh{\phi_\alpha}(k)=1$ for all $k\in\Z\setminus\br{0}$, and $\dist_\infty(\Ds{\alpha}\phi_\alpha, B_1^1)=1$.

The same holds for $\alpha=s\in\N$, $s\ge 2$ with the Riesz derivative $\Ds{s}\phi_s$ replaced by the ordinary derivative $\phi_s^{(s)}$.
\end{lemma}

We postpone the construction of such a function, and show first how this function can serve as an extremal function in Theorem~\ref{thm_main_0}.

\begin{theorem}
a) For each $s\in\N$, $s\ge 2$, and each $\sigma>0$ there exists a function $\phi_{s,\sigma}$ satisfying the hypotheses of Theorem~\ref{thm_main_0}\,a) such that
\begin{equation}\label{eq_main_1a_=}
  \bigabs{R_\sigma(\phi_{s,\sigma})} = h^s\frac{2\zeta(s)}{(2\pi)^{s-1/2}}\dist_\infty\pbig{\phi_{s,\sigma}^{(s)}, B_\sigma^1} \qquad(h=2\pi/\sigma).
\end{equation}

\noindent
b) For each $\alpha>1$ and each $\sigma>0$ there exists a function $\phi_{\alpha,\sigma}$ satisfying the hypotheses of Theorem~\ref{thm_main_0}\,b) such that
\begin{equation}\label{eq_main_1b_=}
  \bigabs{R_\sigma(\phi_{\alpha,\sigma})} = h^s\frac{2\zeta(s)}{(2\pi)^{\alpha-1/2}}\dist_\infty\pbig{\Ds{\alpha}\phi_{\alpha,\sigma}, B_\sigma^1} \qquad(h=2\pi/\sigma).
\end{equation}
\end{theorem}

\begin{proof}
a) First assume $\sigma =1$ and choose $\phi_{s,1}:=\phi_{s}$, the function of Lemma~\ref{prop_existence_extremal} with $\alpha=s$. Then, by definition of the remainder $R_\sigma$, and noting that $\abs{k}^s\wh{\phi_s}(k)=1$ for $k\in\Z\setminus\br{0}$, we have
\begin{align*}
\bigabs{R_1(\phi_s)} = {} & \sqrt{2\pi}\,\biggabs{\,\sideset{}{'}\sum_{k\in\Zprime} \wh{\phi_s}(k)}
      = \sqrt{2\pi}\,\biggabs{\,\sideset{}{'}\sum_{k\in\Zprime} \frac{\abs{k}^s \wh{\phi_s}(k)}{\abs{k}^s}}
{}={}  2\sqrt{2\pi} \, \sum_{k=1}^\infty \rez{k^s}.
\end{align*}
This yields \eqref{eq_main_1a_=} for $\sigma=1$, since $\dist_\infty(\phi_s^{(s)}, B_1^1)=1$.

In order to establish the assertion for arbitrary $\sigma>0$, we choose $\phi_{s,\sigma}(t):=\sigma\phi_{s}(\sigma t)$. Since $\wh{\phi_{s,\sigma}}(v):=\wh{\phi_{s}}(v/\sigma)$, it follows that
\begin{equation}\label{eq_R_vergleich}
  R_\sigma(\phi_{s,\sigma})=-\sqrt{2\pi}\sideset{}{'}\sum_{k\in\Zprime}\wh{\phi_{s,\sigma}}(k\sigma)
  =-\sqrt{2\pi}\sideset{}{'}\sum_{k\in\Zprime}\wh{\phi_{s}}(k)=R_1(\phi_{s}).
\end{equation}
Furthermore, one has
\begin{equation}\label{eq_dist_vergleich}
\begin{aligned}
  \dist_\infty\pbig{\phi_{s,\sigma}^{(s)}, B_\sigma^1}={}&\sup_{\abs{v}\ge\sigma}\brBig{\abs{v}^s \absbig{\wh{\phi_{s,\sigma}}(v)}}
  =\sigma^s\sup_{\abs{v}\ge\sigma}\brbigg{\absBig{\frac v\sigma}^s \absBig{\wh{\phi_{s}}\pBig{\frac v\sigma}}}\\[2ex]
  ={}&\sigma^s\sup_{\abs{u}\ge 1}\brBig{\abs{u}^s \absbig{\wh{\phi_{s}}(u)}}=\sigma^s \dist_\infty\pbig{\phi_{s}^{(s)}, B_1^1}.
\end{aligned}
\end{equation}

Having already proved \eqref{eq_main_1a_=} for $\sigma=1$, \ie $h=2\pi$, with $\phi_{s,\sigma}=\phi_s$, we now can deduce the general case using \eqref{eq_R_vergleich} and \eqref{eq_dist_vergleich}, indeed,
\[
   \bigabs{R_\sigma(\phi_{s,\sigma})} = \bigabs{R_1(\phi_{s})}= 2\sqrt{2\pi}\zeta(s)\dist_\infty\pbig{\phi_{s}^{(s)}, B_1^1}
   =\frac{2\sqrt{2\pi}\zeta(s)}{\sigma^{s}}\dist_\infty\pbig{\phi_{s,\sigma}^{(s)}, B_\sigma^1}.
\]
The proof of a) is completed by substituting $\sigma=2\pi/h$, and b) follows exactly along the same lines.
\end{proof}

We complete our considerations upon extremal functions in Theorem~\ref{thm_main_0} with

{\renewcommand{\proofname}{Proof of Lemma~\ref{prop_existence_extremal}}
\begin{proof}
Let  $\alpha >1$, and let $\sum'_{j\in\Zprime}a_j$ be a convergent series with $0< a_j \le (\alpha+1)^{-1}$, $a_j=a_{-j}$ and $j^\alpha a_j =\Oh(1)$, $j\to\infty$.

Now consider the functions $\psi_j\colon \R\to\C$,
\[
  \psi_j(t):=\frac{a_j}{\sqrt{2\pi}} \sinc^2\pBig{\frac{a_jt}{2\pi}}e^{ijt}\qquad \pbig{j\in\Zprime},
\]
where $\sinc t:= (\sin \pi t)/(\pi t)$ for $t\neq0$ and $\sinc 0 :=1$.
%
The functions $\psi_j$ have the following properties:
\begin{enumerate}[(i),widest*=8,labelsep*=5pt]
\item $\ds \norm{\psi_j}_{\Lone}=\sqrt{2\pi}$;\label{enum_1}
\item $\ds \wh{\psi_j}(v)= \pbigg{1- \frac{\abs{v-j}}{a_j}}_+$, where $u_+=u$ for $u\ge 0$, and $u_+=0$ for $u<0$;\label{enum_2}
\item $\supp \wh{\psi_j}= I_n:=\bkbig{j-a_j,j+a_j}\subset \pbig{j-\frac{1}{2},j+\frac{1}{2}}$;\label{enum_3}
\item $\ds \bignorm{\wh{\psi_j}}_{\Lone}
      =a_j$;\label{enum_4}
\item $\ds \wh{\psi_j}(k)=\delta_{j,k}\quad (j,k\in\Zpprime)$;\label{enum_5}
\item $\ds \wh{\psi_j}\in W^{1,1}(\R)\cap\CRb$ with $\Bignorm{\wh{\psi_j}'}_{\Lone}=\Bignorm{\wh{\psi_j}'}_{L^1(I_j)}=2$;\label{enum_6}
\item $\ds \abs{v}^\alpha \wh{\psi_j}(v)\in \Lone$ with $\Bignorm{\abs{v}^\alpha \wh{\psi_j}(v)}_{\Lone}
            \le \pbig{\abs{j}+a_j}^\alpha a_j =\Oh(1),\ \abs{j}\to\un$;\label{enum_7}
\item $\ds \abs{v}^\alpha \wh{\psi_j}(v)\in \CRb$ with $\Bignorm{\abs{v}^\alpha \wh{\psi_j}(v)}_{\CRb}
           =\abs{v}^\alpha\wh{\psi_j}(v)\Big|_{v=j}\label{enum_8}
=\abs{j}^\alpha$.
\end{enumerate}

Assertions \ref{enum_1}--\ref{enum_7} follow by elementary calculations. As to \ref{enum_8}, first assume $j\in\N$ and let $g_j(v):=\abs{v}^\alpha \wh{\psi_j}(v)$. Then
\[
  g_j(v)=\begin{cases}
    v^\alpha\pbig{1+\frac{v}{a_j}-\frac{j}{a_j}},\ & v\in[j-a_j,j),\\[1.2ex]
    v^\alpha\pbig{1-\frac{v}{a_j}+\frac{j}{a_j}},\ & v\in[j,j+a_j],\\[1.2ex]
    0, &\text{elsewhere.}
  \end{cases}
\]
We want to show that $0\le g_j(v)\le g_j(j)=\abs{j}^\alpha$.

Obviously, $g_j$ is increasing on $[j-a_j,j]$. The assertion is proved, if we can show that $g_j$ is decreasing on $[j,j+a_j]$. In this respect we have
\[
  g_j'(v)=\frac{v^{\alpha-1}}{a_j}\pbig{\alpha a_j - (\alpha+1)v+\alpha j}\le \frac{v^{\alpha-1}}{a_j}\pbig{\alpha a_j - (\alpha+1)j+\alpha j}
  = \frac{v^{\alpha-1}}{a_j}\pbig{\alpha a_j - j}.
\]
Now, since $\alpha a_j \le \alpha/(\alpha+1)<1$ and $j\ge 1$, it follows that $g_j'(v)<0$ on the interval in question, and hence $g_j$ is decreasing there, what was to be shown. If $j<0$, the result follows by noting that $g_j(v)=g_{-j}(-v)$.
\medskip

Next define
\begin{equation}\label{eq_phi_def}
  \phi(t):=\sideset{}{'}\sum_{j\in\Zprime}  \frac{\psi_j(t)}{\abs{j}^\alpha}\qquad (t\in\R).
\end{equation}
Since
\[
  \sideset{}{'}\sum_{j\in\Zprime} {\frac{\abs{\psi_j(t)}}{\abs{j}^\alpha}}
  \le \frac1{\sqrt{2\pi}}\sideset{}{'}\sum_{j\in\Zprime} \frac{a_j}{\abs{j}^{\alpha}}<\infty, \qquad
  \sideset{}{'}\sum_{j\in\Zprime} \biggnorm{\frac{\psi_j}{\abs{j}^{\alpha}}}_{\Lone}
   = \sqrt{2\pi}\sideset{}{'}\sum_{j\in\Zprime} \frac{1}{\abs{j}^{\alpha}}<\un,
\]
it follows that the series in \eqref{eq_phi_def} converges uniformly as well as in $\Lone$, and hence $\phi\in \Lone\cap \CRb$.

Moreover, since the Fourier transform is a bounded operator from $\Lone$ to $\CRn$, one has by (ii),
\begin{equation}\label{eq_phi_dach}
  \wh\phi(v)=\sideset{}{'}\sum_{j\in\Zprime} \frac{1}{\abs{j}^{\alpha}} \wh{\psi_j}(v)= \sideset{}{'}\sum_{j\in\Zprime} \frac{1}{\abs{j}^{\alpha}}\pbigg{1- \frac{\abs{v-j}}{a_j}}_+
  \qquad (v\in\R),
\end{equation}
where the supports of the functions in the infinite series are pairwise disjoint. This enables us to transfer properties (iv)--(vi) from $\wh{\psi_j}$ to $\wh\phi$. Specifically, these include
\begin{enumerate}[(i),widest*=8,labelsep*=5pt,resume*]
\item $\ds \wh\phi(k)= \frac{1}{\abs{k}^{\alpha}}\qquad (k\in \Zpprime)$;\label{enum_9}
%

\item $\ds \wh{\phi}\in W^{1,1}(\R)\cap\CRb$;\label{enum_10}

\item \label{enum_11} $\ds \abs{v}^\alpha \wh{\phi}(v)\in \Lone\cap\CRb\subset \Ltwo\cap\CRb$ with \[\Bignorm{\abs{v}^\alpha \wh{\phi}(v)}_{\CRb}=\abs{v}^\alpha\wh{\phi}(v)\Big|_{v=1}=1.\]
\end{enumerate}
Furthermore, it follows from \ref{enum_11} that $\Ds{\alpha}\phi$ exists as an element of $\Ltwo$ together with $\dist_\infty(\Ds{\alpha}\phi, B_1^1)=1$. All in all, $\phi$ has the properties required for the function $\phi_\alpha$ in Lemma~\ref{prop_existence_extremal}.
\end{proof}
}
%
%
\section{Numerical examples}
In this section we will take formula \eqref{eq_qudrature_formula} as a quadrature formula, \ie,
\begin{equation}\label{eq_quadrature_formula}
  \int_{-\infty}^\infty f(t)\dt \approx h \sum_{k\in\Z} f(hk).
\end{equation}
Of course, in practice one has to replace the infinite series by a finite one. Hence \eqref{eq_quadrature_formula} will be replaced by
\begin{equation}\label{eq_quadrature_finite}
  \int_{-\infty}^\infty f(t)\dt \approx h \sum_{k=-N}^N f(hk)=:S
\end{equation}
for a suitably large $N\in\N$. This leads to the so-called truncation error
\[
T_N(f):=\sum_{|k|>N} f(hk).
\]

In the following examples we always choose $N$ so large that this error is far beyond the precision used in the computations and so can be neglected.

The computations below are performed with MAPLE~16 using 20 decimal places (Digits${}:={}$20).

\subsection*{Example 1}
First we consider the function
\[
    f_1(x):=e^{-\abs{x}}\qquad(x\in\R)
\]
and compute the series in \eqref{eq_quadrature_finite} for different values of $h>0$ with $N:=100/h$. The truncation error can be easily estimated by
\begin{equation*}
  \absbig{T_N(f_1)}\le 2\int_{t\ge N}e^{-ht}\dt=\frac{2e^{-hN}}{h}\le 8\cdot10^{-41}\qquad (h\ge 10^{-3})
\end{equation*}
and will be neglected in the following.

Table~\ref{table_1} below shows the current value of $h$, the value of the series $S$ in \eqref{eq_quadrature_finite} and the associated error
\[
  E:= \biggabs{S- \int_{-\infty}^\infty f_1(t)\dt}=\abs{S-2}.
\]
According to Theorem~\ref{thm_main_0}, this error can be estimated in terms of the distance functional via \eqref{eq_main_1_alpha}. Since the Fourier transform of $f_1$ is given by
\[
  \widehat{f_1}(v)= \frac{\sqrt {2}}{\sqrt {\pi }( 1+{v}^{2})}\qquad(v\in\R),
\]
one can apply Theorem~\ref{thm_main_0}\,b) with any $\alpha\in (1,\frac32)$ to give
\[
  E=\Oh(h^2)\qquad(h\to 0+).
\]
This order is confirmed in the rightmost column of Table~\ref{table_1}. Indeed, the entries $E/h^2$ are approximately constant.

\tabcolsep=15pt
\renewcommand{\arraystretch}{1.2}

\begin{table}[H]\centering
{\ttfamily

\begin{tabular}{|c|c|c|c|}
$h$ & $S$ & $E$ & $E/h^2$\\\hline
 2.000  &  2.626070570998663   &   6.3e-01  &  0.1565 \\
 1.000  &  2.163953413738653   &   1.6e-01  &  0.1640 \\
 0.800  &  2.105545953465751   &   1.1e-01  &  0.1649 \\
 0.600  &  2.059643058193045   &   6.0e-02  &  0.1657 \\
 0.400  &  2.026595825375789   &   2.7e-02  &  0.1662 \\
 0.200  &  2.006662226450798   &   6.7e-03  &  0.1666 \\
 0.100  &  2.001666388955010   &   1.7e-03  &  0.1666 \\
 0.080  &  2.001066552906224   &   1.1e-03  &  0.1666 \\
 0.060  &  2.000599964003085   &   6.0e-04  &  0.1667 \\
 0.040  &  2.000266659555826   &   2.7e-04  &  0.1667 \\
 0.020  &  2.000066666222226   &   6.7e-05  &  0.1667 \\
 0.010  &  2.000016666638889   &   1.7e-05  &  0.1667 \\
 0.008  &  2.000010666655289   &   1.1e-05  &  0.1667 \\
 0.006  &  2.000005999996400   &   6.0e-06  &  0.1667 \\
 0.004  &  2.000002666665956   &   2.7e-06  &  0.1667 \\
 0.002  &  2.000000666666622   &   6.7e-07  &  0.1667 \\
 0.001  &  2.000000166666664   &   1.7e-07  &  0.1667 \\
\end{tabular}}\caption{Table}\label{table_1}
\end{table}

\subsection*{Example 2}
It follows from Theorem~\ref{thm_main_0} that the error $E$ does not depend on the decay of the function involved but on the decay of its Fourier transform. This will also be confirmed by our next example
\[
    f_2(x):=x^2 e^{-\abs{x}}\qquad(x\in\R)
\]
having Fourier transform
\[
  \widehat{f_2}(v)=\frac{2\sqrt{2}(-3v^{2}+1)}{\sqrt{\pi}(1+v ^{2})^{3}}\qquad (v\in\R).
\]

Choosing $N:=100/h$, as above, and noting that $f_2(x)$ is non-increasing for $x\ge 2$, we obtain for the truncation error
\begin{equation*}
  \absbig{T_N(f_2)}\le 2\int_{t\ge N}t^2 e^{-ht}\dt\le 8\cdot10^{-37}\qquad (h\ge 10^{-3}).
\end{equation*}

By Theorem~\ref{thm_main_0}\,a) or b) with, \eg, $s=\alpha=2$, the expected order of the error $E$ is
\[
  E= \biggabs{S- \int_{-\infty}^\infty f_2(t)\dt}=\abs{S-4}=\Oh(h^4)\qquad (h\to 0+),
\]
which is approved by the entries in Column~4 of Table~\ref{table_2}.

\tabcolsep=15pt
\renewcommand{\arraystretch}{1.2}

\begin{table}[H]\centering
{\ttfamily

\begin{tabular}{|c|c|c|c|}
$h$ & $S$ & $E$ & $E/h^4$\\\hline
2.000  &  3.802874038904078   &   2.0e-01  &  0.0123 \\
 1.000  &  3.984589534249975   &   1.5e-02  &  0.0154 \\
 0.800  &  3.993508748796519   &   6.5e-03  &  0.0158 \\
 0.600  &  3.997900565813553   &   2.1e-03  &  0.0162 \\
 0.400  &  3.999578706124929   &   4.2e-04  &  0.0165 \\
 0.200  &  3.999973417811948   &   2.7e-05  &  0.0166 \\
 0.100  &  3.999998334655391   &   1.7e-06  &  0.0167 \\
 0.080  &  3.999999317679968   &   6.8e-07  &  0.0167 \\
 0.060  &  3.999999784061703   &   2.2e-07  &  0.0167 \\
 0.040  &  3.999999957338751   &   4.3e-08  &  0.0167 \\
 0.020  &  3.999999997333418   &   2.7e-09  &  0.0167 \\
 0.010  &  3.999999999833335   &   1.7e-10  &  0.0167 \\
 0.008  &  3.999999999931734   &   6.8e-11  &  0.0167 \\
 0.006  &  3.999999999978400   &   2.2e-11  &  0.0167 \\
 0.004  &  3.999999999995733   &   4.3e-12  &  0.0167 \\
 0.002  &  3.999999999999733   &   2.7e-13  &  0.0167 \\
 0.001  &  3.999999999999983   &   1.7e-14  &  0.0167 \\
\end{tabular}}\caption{Table}\label{table_2}
\end{table}

\subsection*{Example 3}
Our last example is the infinitely smooth function
\[
    f_3(x):=\frac{1}{1+x^6}\qquad(x\in\R)
\]
with
\[
  \int_{-\infty}^\infty f_2(t)\dt=\frac{2\pi}{3}=2.094395102393195\dots
\]

In order to have the truncation error sufficiently small, we choose $N:=10^4/h$, and obtain
%
%
 $ \absbig{T_N(f_3)}\le 2\int_{t\ge N}t^{-6}\dt\le 7\cdot10^{-20}$ for $h\ge 6\cdot 10^{-2}$.
%
%

In view of estimate \eqref{eq_R_mit_ableitung}, it is to be expected that the error $E$ decreases faster than any power of $h$ for $h\to0+$. Of course, this cannot be verified by a numerical example, but it can be seen in Table~\ref{table_3} that $E$ decreases more rapidly than the error in Table~\ref{table_1} or~\ref{table_2}.

\begin{table}[H]\centering
{\ttfamily

\begin{tabular}{|c|c|c|}
$h$ & $S$ & $E$ \\\hline
 2.00  &  2.062622162651823   &   3.2e-02   \\
 1.00  &  2.034201468066433   &   6.0e-02   \\
 0.80  &  2.168128178972554   &   7.4e-02   \\
 0.60  &  2.090709297829778   &   3.7e-03   \\
 0.40  &  2.096021763767048   &   1.6e-03   \\
 0.20  &  2.094395428970767   &   3.3e-07   \\
 0.10  &  2.094395102393100   &   9.5e-14   \\
 0.08  &  2.094395102393195   &   2.0e-17   \\
 0.06  &  2.094395102393195   &   5.0e-19
\end{tabular}}\caption{Table}\label{table_3}
\end{table}


\section{A short biography of Helmut Brass 1936--2011}\label{sec_Brass}
Helmut Brass (originally written as Braß) was born in Hannover, Germany, in 1936. After completing Oberrealschule, he worked in a firm that produced and recycled copper cables. Since already as a schoolboy he was very interested in chemistry, he then enrolled for this subject at the University of Hannover. But soon after he had started, he realized that his true talent was mathematics and turned to it. In 1962 he graduated with a diploma in mathematics and continued as a Scientific Assistant under the supervision of Wilhelm Quade. The latter was known as a pioneer of spline functions with a contribution (jointly with L.~Collatz) as early as 1938. In 1965 Brass received his doctoral degree in mathematics with a thesis on approximation by a linear combination of projection operators. In 1968 he acquired Habilitation and became a University Dozent.

In 1970, Brass was appointed as a professor at the Technical University of Clausthal and in 1974 he followed the offer of a chair at the University of Osnabrück, but already in 1977 he accepted a chair at the University of Braunschweig where he stayed until his retirement in 2002.

The research field of Helmut Brass comprised interpolation and approximation with special emphasis on quadrature. In particular, he studied optimal and nearly optimal quadrature formulae for various classes of functions, properties of the remainder functional such as positivity and monotonicity, best or asymptotically best error estimates for classical quadrature formulae and exact rates of convergence under side conditions on the involved function such as periodicity, convexity or bounded variation. He published about 50 research papers and two distinguished books: Quadraturverfahren in 1977 \cite{Brass_1977} and (jointly with K.\ Petras) Quadrature Theory in 2011. His style was to follow the naturally arising questions systematically and either come to a convincing answer or end up with an open problem. This way, his book of 1977 has inspired several young mathematicians for further research. Brass also edited two Proceedings of Oberwolfach Conferences on numerical integration. Furthermore he wrote a fascinating booklet with eleven lectures on Bernoulli polynomials designed  for the training of students in \emph{Pro\-semi\-nars}.

Some of Brass's excellent results are not easily accessible since he communicated them (with complete proofs) in conference proceedings or reports of scientific societies and in German language, thinking that his English may not be good enough for an attractive presentation. For the same reason, he rarely went to conferences abroad. He lectured on his results on meetings of national societies, on international conferences inside Germany and on several Oberwolfach meetings conducting two of them himself.

Brass had twelve research students who graduated with a doctoral degree under his supervision. Four of them acquired the Habilitation degree and one was also awarded with the title of a University Professor.

Apart from his scientific activities, Brass was much engaged in administration and university politics. For quite some time he was member of the senate of his university. He also served as a dean and a chairman.

In 1963, Brass married Gisela Lueder. They had studied together in Hannover. She was a Gymnasium teacher of mathematics and chemistry. They had two sons, Stefan and Peter, both now being professors of computer science, one in Halle (Germany), and the other in New York. Peter is also known as a mathematician with impressive contributions to discrete geometry including the solution of a problem of Paul Erd\H{o}s.

In 2008 a stroke of fate met the whole family, when Mrs.~Gisela Brass died all of a sudden. Helmut Brass never recovered from this shock. His health deteriorated. He had several stays in hospitals and rehabilitation centers. He passed away in Halle on October~30, 2011 in the house of his elder son Stefan.

Already over more than two decades Brass had, in cooperation with his student Knut Petras, collected material and drafted, revised and updated a manuscript for a new book on quadrature. Their aim was a systematical approach to error estimates based on properties of the function $f$ to be integrated and hence holding for any other function $g$ as well that shares these properties with $f$. After the death of his wife, Brass realized that he must concentrate on finalizing this project or else it may be lost for ever. His new book appeared as a publication of the American Mathematical Society one week after his passing.

\bigskip

I (G.\,S.) met Helmut Brass for the first time in Oberwolfach in November 1977 at a conference on Numerical Methods in Approximation Theory and again in Oberwolfach in 1978, 1981, 1987, 1992 and 2001 at conferences on Numerical Integration, as well as on a few other occasions.

A few months preceding our first meeting in 1977, Brass' book \cite{Brass_1977} had appeared. I was impressed by its systematic composition and the wealth of results. Although I had already published two papers on quadrature (\cite{Schmeisser_1972} and \cite{Schmeisser_1977}), I felt that I could not consider myself to be an expert. I was happy to see that the book contained a subsection on the theme of \cite{Schmeisser_1972} where my paper was quoted and appreciated, while \cite{Schmeisser_1977} had not appeared before the book was completed. I was happy about this too, because I discovered in Brass' book that my result in \cite{Schmeisser_1977} was already known, although my proof was new. Eagerly I looked for open problems, found on page 141 that the definiteness of the remainders of Filippi's formulae for even order had not yet been established and realized that the method of my worthless paper \cite{Schmeisser_1977} would work after some adaptions. My proof used splines. When I sent it to Brass, he praised me for having filled an ugly gap but also mentioned that my proof can probably be simplified by avoiding splines and be extended to a wider class of problems. Over a period of more than half a year, we exchanged ideas in long letters, gained deeper insight and arrived at an elegant and perfect proof. For me it was clear that Brass must be my co-author. He hesitated and played down his role, but finally he agreed \cite{Brass-Schmeisser_1979} under the condition that we write a second joint paper in which he would have the chance to contribute the fundamental idea. In fact, he had a strikingly simple idea for a powerful comparison technique of linear functionals for which we jointly worked out many applications on interpolatory quadrature formulae \cite{Brass-Schmeisser_1981}.

Later I profited from work of Brass \cite{Brass_1978} in my collaboration with Q.\,I.~Rahman, when we characterized the speed of convergence of the trapezoidal formula and related quadrature methods in terms of function spaces; see, e.g., \cite{Rahman-Schmeisser_1990}.

In February 1979 Brass visited me in Erlangen and gave a talk in our Mathematical Colloquium. He reciprocated by inviting me to a colloquium talk in Braunschweig in May 1981 and to an evening in his house. I appreciated Brass as a warm person with a quick mind, a wide knowledge and clever ideas.

\bigskip

I (P.L.\,B.) invited Wilhelm Quade (1898--1975), a discoverer of splines, to my first conference at Oberwolfach (1963). He brought along with him his research assistant Helmut Brass, a shy young mathematician who was to complete his doctorate under Quade in 1965. There he had the opportunity to get to know in person a great number of  renowned mathematicians from many different countries.

Being aware of his basic work in quadrature methods and his unique book \cite{Brass_1977}, I invited him on the recommendation of my colleague Rolf Nessel to my Oberwolfach conference in 1980.

The third and final time I met him was ca.~1988 when he invited me to give a colloquium lecture at the Technische Hochschule Braunschweig. Although I had systematically turned down all invitations to colloquium talks in Germany for some twenty past years -- at the time I had thought that it is more important to take care of my many master and  doctoral students at Aachen -- I accepted the kind Brass invitation.

The uniform boundedness principle (UBP) of functional analysis was a chief area of research of my colleague Rolf Nessel in the eighties, actually an area which Rolf opened up.
Together with several students, especially Erich van Wickeren and Werner Dickmeis, he wrote several papers in the broad area; for a survey type paper and a book-type presentation see \cite{Dickmeis-Nessel_1982,Dickmeis-Nessel-Wickeren_1987}. They dealt with qualitative extensions of the UBP and with the sharpness of error estimates. Their applications of this general principle to quadrature, in particular to the trapezoidal rule, were motivated by three papers of Helmut Brass, in which he gave best possible error estimates for quadrature rules \cite{Brass_1973,Brass_1978,Brass_1979}. It was especially Brass's book \cite{Brass_1977} which was their home base in understanding quadrature formulae, in fact the best book Rolf found.

\bigskip
We believe that our present paper is a good spot for commemorating  Helmut Brass. It is fully in his spirit since it considers classical formulae and approaches them with new ideas that lead to new, asymptotically best possible error estimates.

We wish to thank Peter Brass for sending us some material for the biographical note.

\bibliographystyle{elsarticle-num}
\bibliography{Poisson}

\end{document}